\documentclass[12pt,reqno]{amsart} 
\usepackage{amssymb,amscd,url}

\begin{document}

\allowdisplaybreaks



\newtheorem{theorem}{Theorem}
\newtheorem{lemma}[theorem]{Lemma}
\newtheorem{conjecture}[theorem]{Conjecture}
\newtheorem{proposition}[theorem]{Proposition}
\newtheorem{corollary}[theorem]{Corollary}
\newtheorem*{claim}{Claim}

\theoremstyle{definition}
\newtheorem{question}{Question}
\renewcommand{\thequestion}{\Alph{question}} 
\newtheorem*{definition}{Definition}
\newtheorem{example}[theorem]{Example}

\theoremstyle{remark}
\newtheorem{remark}[theorem]{Remark}
\newtheorem*{acknowledgement}{Acknowledgements}



\newenvironment{notation}[0]{%
  \begin{list}%
    {}%
    {\setlength{\itemindent}{0pt}
     \setlength{\labelwidth}{4\parindent}  
     \setlength{\labelsep}{\parindent}
     \setlength{\leftmargin}{5\parindent}
     \setlength{\itemsep}{0pt}
     }%
   }%
  {\end{list}}

\newenvironment{parts}[0]{%
  \begin{list}{}%
    {\setlength{\itemindent}{0pt}
     \setlength{\labelwidth}{1.5\parindent}
     \setlength{\labelsep}{.5\parindent}
     \setlength{\leftmargin}{2\parindent}
     \setlength{\itemsep}{0pt}
     }%
   }%
  {\end{list}}
\newcommand{\Part}[1]{\item[\upshape#1]}

\renewcommand{\a}{\alpha}
\renewcommand{\b}{\beta}
\newcommand{\g}{\gamma}
\renewcommand{\d}{\delta}
\newcommand{\e}{\epsilon}
\newcommand{\ve}{\varepsilon}
\newcommand{\f}{\varphi}
\newcommand{\bfphi}{{\boldsymbol{\f}}}
\renewcommand{\l}{\lambda}
\renewcommand{\k}{\kappa}
\newcommand{\lhat}{\hat\lambda}
\newcommand{\m}{\mu}
\newcommand{\bfmu}{{\boldsymbol{\mu}}}
\renewcommand{\o}{\omega}
\renewcommand{\r}{\rho}
\newcommand{\rbar}{{\bar\rho}}
\newcommand{\s}{\sigma}
\newcommand{\sbar}{{\bar\sigma}}
\renewcommand{\t}{\tau}
\newcommand{\z}{\zeta}

\newcommand{\D}{\Delta}
\newcommand{\G}{\Gamma}
\newcommand{\F}{\Phi}

\newcommand{\ga}{{\mathfrak{a}}}
\newcommand{\gA}{{\mathfrak{A}}}
\newcommand{\gb}{{\mathfrak{b}}}
\newcommand{\gB}{{\mathfrak{B}}}
\newcommand{\gc}{{\mathfrak{c}}}
\newcommand{\gC}{{\mathfrak{C}}}
\newcommand{\gm}{{\mathfrak{m}}}
\newcommand{\gn}{{\mathfrak{n}}}
\newcommand{\go}{{\mathfrak{o}}}
\newcommand{\gO}{{\mathfrak{O}}}
\newcommand{\gp}{{\mathfrak{p}}}
\newcommand{\gP}{{\mathfrak{P}}}
\newcommand{\gq}{{\mathfrak{q}}}
\newcommand{\gR}{{\mathfrak{R}}}
\newcommand{\gU}{{\mathfrak{U}}}

\newcommand{\Abar}{{\bar A}}
\newcommand{\Ebar}{{\bar E}}
\newcommand{\Kbar}{{\bar K}}
\newcommand{\Pbar}{{\bar P}}
\newcommand{\Sbar}{{\bar S}}
\newcommand{\Tbar}{{\bar T}}
\newcommand{\ybar}{{\bar y}}
\newcommand{\phibar}{{\bar\f}}

\newcommand{\Acal}{{\mathcal A}}
\newcommand{\Bcal}{{\mathcal B}}
\newcommand{\Ccal}{{\mathcal C}}
\newcommand{\Dcal}{{\mathcal D}}
\newcommand{\Ecal}{{\mathcal E}}
\newcommand{\Fcal}{{\mathcal F}}
\newcommand{\Gcal}{{\mathcal G}}
\newcommand{\Hcal}{{\mathcal H}}
\newcommand{\Ical}{{\mathcal I}}
\newcommand{\Jcal}{{\mathcal J}}
\newcommand{\Kcal}{{\mathcal K}}
\newcommand{\Lcal}{{\mathcal L}}
\newcommand{\Mcal}{{\mathcal M}}
\newcommand{\Ncal}{{\mathcal N}}
\newcommand{\Ocal}{{\mathcal O}}
\newcommand{\Pcal}{{\mathcal P}}
\newcommand{\Qcal}{{\mathcal Q}}
\newcommand{\Rcal}{{\mathcal R}}
\newcommand{\Scal}{{\mathcal S}}
\newcommand{\Tcal}{{\mathcal T}}
\newcommand{\Ucal}{{\mathcal U}}
\newcommand{\Vcal}{{\mathcal V}}
\newcommand{\Wcal}{{\mathcal W}}
\newcommand{\Xcal}{{\mathcal X}}
\newcommand{\Ycal}{{\mathcal Y}}
\newcommand{\Zcal}{{\mathcal Z}}

\renewcommand{\AA}{\mathbb{A}}
\newcommand{\BB}{\mathbb{B}}
\newcommand{\CC}{\mathbb{C}}
\newcommand{\FF}{\mathbb{F}}
\newcommand{\GG}{\mathbb{G}}
\newcommand{\NN}{\mathbb{N}}
\newcommand{\PP}{\mathbb{P}}
\newcommand{\QQ}{\mathbb{Q}}
\newcommand{\RR}{\mathbb{R}}
\newcommand{\ZZ}{\mathbb{Z}}

\newcommand{\bfa}{{\mathbf a}}
\newcommand{\bfb}{{\mathbf b}}
\newcommand{\bfc}{{\mathbf c}}
\newcommand{\bfe}{{\mathbf e}}
\newcommand{\bff}{{\mathbf f}}
\newcommand{\bfg}{{\mathbf g}}
\newcommand{\bfp}{{\mathbf p}}
\newcommand{\bfr}{{\mathbf r}}
\newcommand{\bfs}{{\mathbf s}}
\newcommand{\bft}{{\mathbf t}}
\newcommand{\bfu}{{\mathbf u}}
\newcommand{\bfv}{{\mathbf v}}
\newcommand{\bfw}{{\mathbf w}}
\newcommand{\bfx}{{\mathbf x}}
\newcommand{\bfy}{{\mathbf y}}
\newcommand{\bfz}{{\mathbf z}}
\newcommand{\bfA}{{\mathbf A}}
\newcommand{\bfF}{{\mathbf F}}
\newcommand{\bfB}{{\mathbf B}}
\newcommand{\bfD}{{\mathbf D}}
\newcommand{\bfG}{{\mathbf G}}
\newcommand{\bfI}{{\mathbf I}}
\newcommand{\bfM}{{\mathbf M}}
\newcommand{\bfzero}{{\boldsymbol{0}}}

\newcommand{\Adele}{\textsf{\upshape A}}
\newcommand{\Ahat}{\hat{A}}
\newcommand{\Adot}{A(\Adele_K)_{\bullet}}
\newcommand{\Aut}{\operatorname{Aut}}
\newcommand{\Br}{\operatorname{Br}}  
\newcommand{\Closure}{\textsf{\upshape C}} 
\newcommand{\Disc}{\operatorname{Disc}}
\newcommand{\Div}{\operatorname{Div}}
\newcommand{\End}{\operatorname{End}}
\newcommand{\Fbar}{{\bar{F}}}
\newcommand{\FOD}{\textup{FOM}}
\newcommand{\FOM}{\textup{FOD}}
\newcommand{\Gal}{\operatorname{Gal}}
\newcommand{\GL}{\operatorname{GL}}
\newcommand{\Index}{\operatorname{Index}}
\newcommand{\Image}{\operatorname{Image}}
\newcommand{\liftable}{{\textup{liftable}}}
\newcommand{\hhat}{{\hat h}}
\newcommand{\Ksep}{K^{\textup{sep}}}
\newcommand{\Ker}{{\operatorname{ker}}}
\newcommand{\Lsep}{L^{\textup{sep}}}
\newcommand{\Lift}{\operatorname{Lift}}
\newcommand{\LS}[2]{{\genfrac{(}{)}{}{}{#1}{#2}}} 
\newcommand{\vlim}{\operatornamewithlimits{\text{$v$}-lim}}
\newcommand{\wlim}{\operatornamewithlimits{\text{$w$}-lim}}
\renewcommand{\max}{\operatornamewithlimits{max}}
\newcommand{\MOD}[1]{~(\textup{mod}~#1)}
\newcommand{\Norm}{{\operatorname{\mathsf{N}}}}
\newcommand{\norm}[1]{\ensuremath{\| #1 \|}}
\newcommand{\notdivide}{\nmid}
\newcommand{\normalsubgroup}{\triangleleft}
\newcommand{\odd}{{\operatorname{odd}}}
\newcommand{\onto}{\twoheadrightarrow}
\newcommand{\Orbit}{\mathcal{O}}
\newcommand{\ord}{\operatorname{ord}}
\newcommand{\Per}{\operatorname{Per}}
\newcommand{\PrePer}{\operatorname{PrePer}}
\newcommand{\PGL}{\operatorname{PGL}}
\newcommand{\Pic}{\operatorname{Pic}}
\newcommand{\Prob}{\operatorname{Prob}}
\newcommand{\Qbar}{{\bar{\QQ}}}
\newcommand{\rank}{\operatorname{rank}}
\newcommand{\Rat}{\operatorname{Rat}}
\newcommand{\reduction}[1]{\ensuremath{\widetilde{#1}}}
\newcommand{\Resultant}{\operatorname{Res}}
\renewcommand{\setminus}{\smallsetminus}
\newcommand{\Span}{\operatorname{Span}}
\newcommand{\Support}{\operatorname{Support}}
\newcommand{\tors}{{\textup{tors}}}
\newcommand{\Trace}{\operatorname{Trace}}
\newcommand{\twistedtimes}{\mathbin{%
   \mbox{$\vrule height 6pt depth0pt width.5pt\hspace{-2.2pt}\times$}}}
\newcommand{\UHP}{{\mathfrak{h}}}    
\newcommand{\Vdot}{V(\Adele_K)_{\bullet}}
\newcommand{\Wreath}{\operatorname{Wreath}}
\newcommand{\<}{\langle}
\renewcommand{\>}{\rangle}

\newcommand{\longhookrightarrow}{\lhook\joinrel\longrightarrow}
\newcommand{\longonto}{\relbar\joinrel\twoheadrightarrow}

\newcommand{\Spec}{\operatorname{Spec}}
\renewcommand{\div}{{\operatorname{div}}}

\newcounter{CaseCount}
\Alph{CaseCount}
\def\Case#1{\par\vspace{1\jot}\noindent
\stepcounter{CaseCount}
\framebox{Case \Alph{CaseCount}.\enspace#1}
\par\vspace{1\jot}\noindent\ignorespaces}

\newcommand{\helpme}[1]{{\sf $\heartsuit\heartsuit$ Meta-remark: [#1]}}
\newcommand{\fixme}[1]{{\sf $\spadesuit\spadesuit$ Meta-remark: [#1]}}

\title[Integral points in orbits]{
   A quantitative estimate for quasi-integral points in orbits}
\date{\today}

\author{Liang-Chung Hsia and Joseph H. Silverman}
\address{Department of Mathematics,
         National Central University,
         Chung-Li, 32054 Taiwan, R. O. C.}
\email{hsia@math.ncu.edu.tw}
\address{Mathematics Department, 
         Box 1917,
         Brown University, 
         Providence, RI 02912 USA}
\email{jhs@math.brown.edu}
\thanks{The fist author's research is supported by
  NSC-97-2918-I-008-005 and NSC-96-2115-008-012-MY3.
  The second author's research is supported by NSF DMS-0650017
  and DMS-0854755.}

\begin{abstract}
Let $\f(z)\in K(z)$ be a rational function of degree~$d\ge2$ defined
over a number field whose second iterate~$\f^2(z)$ is not a
polynomial, and let~$\a\in K$.  The second author previously proved
that the forward orbit~$\Orbit_\f(\a)$ contains only finitely many
quasi-$S$-integral points.  In this note we give an explicit upper
bound for the number of such points.
\end{abstract}


\maketitle

\section*{Introduction}
Let~$K/\QQ$ be a number field, let~$S$ be a finite set of places
of~$K$, and let~$1\ge\ve>0$. An element $x\in K$ is said to be
\emph{quasi-$(S,\ve)$-integral} if
\begin{equation}
  \label{eqn:quasiSint}
  \sum_{v\in S}\frac{[K_v:\QQ_v]}{[K :\QQ]}\log^+|x|_v
  \ge \ve h(x).
\end{equation}
We observe that~$x$ is in the ring of~$S$-integers of~$K$ if and only
if it is quasi-$(S,1)$-integral, in which case~\eqref{eqn:quasiSint}
is an equality by definition of the height.
\par
Let $\f(z)\in K(z)$ be a rational function of degree~$d\ge2$,
let~$\a\in K$ be a point, and let
\[
  \Orbit_\f(\a) = \bigl\{\a,\f(\a),\f^2(\a),\ldots\bigr\}
\]
denote the forward orbit of~$\a$ under iteration of~$\f$.
The second author proved in~\cite{silverman_duke} that if~$\f^2(z)$ is
not a polynomial, then the orbit~$\Orbit_\f(\a)$ contains only
finitely many quasi-$(S,\ve)$-integral points. More generally,
if $\#\Orbit_\f(\a)=\infty$ and
if~$\b$ is not an exceptional point for~$\f$, then there
are only finitely many~$n\ge1$ such that 
\[
  \frac{1}{\f^n(\a)-\b}
\]
is quasi-$(S,\ve)$-integral. In this note we give an upper bound for
the number of such~$n$,  making explicit the dependence
on~$S$,~$\f$,~$\a$, and~$\b$. More precisely, we prove that the number
of elements in the set
\begin{equation}
  \label{eqn:nge0fnab}
  \bigl\{n\ge0 : 
  \text{$\bigl(\f^n(\a)-\b\bigr)^{-1}$ is quasi-$(S,\ve)$-integral} \bigr\}
\end{equation}
is smaller than
\begin{equation}
  \label{eqn:4Sgldhfba}
   4^{\#S}\g
      + \log^+_d\left(\frac{h(\f) +  \hhat_\f(\b)}{\hhat_\f(\a)}\right),
\end{equation}
where~$\g$ depends only on~$d$,~$\ve$, and~$[K:\QQ]$.  (See
Section~\ref{subsec:dynamicsheightfunction} for the definitions of the
height~$h(\f)$ of the map~$\f$ and the canonical height~$\hhat_\f$.)
Our main result, Theorem~\ref{thm:main-theorem} in
Section~\ref{sec:main-theorem}, is a strengthened version of this
statement.

The specific form of the upper bound in~\eqref{eqn:4Sgldhfba} is
interesting, especially the dependence on the wandering point~$\a$ and
the target point~$\b$. For example, if~$\hhat_\f(\a)$ is sufficiently
large (depending on~$\b$ and~$\f$), then the bound is independent
of~$\a$,~$\b$, and~$\f$. It is also interesting to ask whether it is
possible, for a given~$\f$ and~$\a$, to make the
set~\eqref{eqn:nge0fnab} arbitrarily large by varying~$\b$.  We
discuss this question further in Remark~\ref{rmk:dynamical_langconj}.

We briefly describe the organization of the paper. We start in
Section~\ref{section:preliminarynotation} by setting notation and
proving an elementary estimate for the chordal metric.
Section~\ref{subsec:dynamicsheightfunction} is devoted to height
functions, both the canonical height associated to a rational map and
various results relating heights and polynomials. In Section~\ref{subsec:dist}
we prove a uniform version of the inverse function theorem for
rational maps of degree~$d$. Section~\ref{subsec:distandroth} states
an estimate for the ramification of the iterate of a rational
function, taken from~\cite{silverman_duke,silverman_book07}, and a
quantititative version of Roth's theorem, taken
from~\cite{silverman_crelles87}. In Section~\ref{sec:main-theorem} we
combine the preliminary material to prove our main theorem.  Finally,
in Section~\ref{sec:bndintpts}, we use the main theorem to give an
explicit upper bound for the number of $S$-integral points in an
orbit.

\begin{remark}
The original paper on finiteness of quasi-$S$-integral points in
orbits~\cite{silverman_duke} has been used by Patrick Ingram and the
second author~\cite{ingram-silverman} to prove a dynamical version of
the classical Bang--Zsigmondy theorem on primitive
divisors~\cite{bang,zsigmondy}.  It has also been used by Felipe
Voloch and the second author~\cite{silverman-voloch09} to prove a
local--global criterion for dynamics on~$\PP^1$.  The quantitative
results proven in the present paper should enable one to prove
quantitative versions of both~\cite{ingram-silverman}
and~\cite{silverman-voloch09}, but we have not included these
applications in this paper in order to keep it to a manageable length.
\end{remark}

\begin{remark}
Quantitative estimates similar to those in this paper have been proven
for the number of integral points on elliptic curves and on certain
other types of curves.  See for example~\cite{gross-silverman95}
and~\cite{silverman_crelles87}.
\end{remark}

\begin{acknowledgement}
The first author would like to thank his coauthor and the Department
of Mathematics at Brown University for their hospitality during his
visit when this work was initiated. The second author would like to
thank Microsoft Research New England for inviting him to be a visiting
researcher.
\end{acknowledgement}

\section{Preliminary Material and Notation}
\label{section:preliminarynotation}

We set the following notation:

\begin{notation}
\setlength{\itemsep}{3pt}
\item[$K$]
a number field;
\item[$M_K$]
the set of places of $K$;
\item[$M_K^\infty$]
the set of archimedean (infinite) places of $K$;
\item[$M_K^0$]
the set of nonarchimedean (finite) places of $K$;
\item[$\log^+(x)$]
the maximum of $\log(x)$ and $0$. We write $\log_d^+$ for log base~$d$.
\end{notation}
\medskip

For each $v \in M_K$, we let $|\,\cdot\,|_v$ denote the corresponding
normalized absolute value on $K$ whose restriction to $\QQ$ gives the
usual $v$-adic absolute value on $\QQ.$ That is, if $v\in
M_K^{\infty}$, then $|x|_v$ is the usual archimedean absolute value,
and if $v\in M_K^0$, then $|x|_v = |x|_p$ is the usual $p$-adic
absolute value for a unique prime $p$. We also write $K_v$ for the
completion of $K$ with respect to $|\,\cdot\,|_v$, and we let $\CC_v$
denote the completion of an algebraic closure of~$K_v$.

For each $v\in M_K$, we let $\r_v$ denote the \emph{chordal metric}
defined on $\PP^1(\CC_v)$, where we recall that for $[x_1, y_1], [x_2,
y_2]\in \PP^1(\CC_v)$,
\[
  \r_v\bigl([x_1, y_1], [x_2, y_2]\bigr) 
  = \begin{cases}
    \dfrac{|x_1y_2- x_2
      y_1|_v}{\sqrt{|x_1|_v^2+ |y_1|_v^2}\sqrt{|x_2|_v^2+|y_2|_v^2}} &
    \text{\!if $v\in M_K^{\infty}$,}
    \\[3\jot]
    \dfrac{|x_1y_2- x_2
      y_1|_v}{\max\{|x_1|_v,|y_1|_v\}\max\{|x_2|_v, |y_2|_v\}} &
    \text{\!if $v\in M_K^0$.}
  \end{cases}
\]

In this paper, we use the logarithmic version of the chordal metric
to measure the distance between points in $\PP^1(\CC_v).$ 

\begin{definition}
The \emph{logarithmic chordal metric function}
\[
  \l_v : \PP^1(\CC_v)\times \PP^1(\CC_v) \to \RR \cup \{\infty\}
\]
is defined by 
\[
 \l_v\bigl([x_1, y_1], [x_2, y_2]\bigr) 
  = - \log \r_v\bigl([x_1, y_1], [x_2, y_2]\bigr).
\]
\end{definition} 

Notice that $\l_v(P, Q) \ge 0$ for all $P, Q\in \PP^1(\CC_v)$, and
that two points $P, Q \in \PP^1(\CC_v)$ are close if and only if
$\l_v(P, Q)$ is large. We also observe that~$\l_v$ is a particular
choice of an \emph{arithmetic distance function} as defined
in~\cite[\S3]{silverman_mathann87}, i.e., it is a local height
function $\l_{\PP^1\times\PP^1,\D}$, where~$\D$ is the diagonal
of~$\PP^1\times\PP^1$.

The next lemma relates the logarithmic chordal metric $\l_v(x,y)$ to
the usual metric $|x - y|_v$ arising from the absolute value~$v$.

\begin{lemma}
\label{lemma:log_chordal}
Let $v\in M_K$  and let $\l_v$ be the logarithmic chordal
metric on $\PP^1(\CC_v).$ Define $\ell_v = 2$ if $v$ is archimedean, and
$\ell_v = 1$ if $v$ is non-archimedean. Then for  $x, y \in  \CC_v$ 
we have
\begin{multline*}
  \l_v(x, y) > \l_v(y, \infty) + \log \ell_v \\
  \Longrightarrow\quad
  \l_v(y, \infty) \le \l_v(x, y)+ \log |x - y|_v  \le 2\l_v(y, \infty)
     + \log \ell_v. 
\end{multline*}
\end{lemma}

\begin{proof}
Notice that by the definition of chordal metric,
\[
  \l_v(x, y) = \l_v(x, \infty) + \l_v(y, \infty) - \log|x-y|_v. 
\]
Therefore,
\[
   \l_v(x, y)+ \log |x - y|_v  = \l_v(x, \infty) + \l_v(y, \infty) \ge
   \l_v(y, \infty).  
\]
This gives the lower bound for the sum $\l_v(x, y)+ \log |x - y|_v.$

For the upper bound, if $v$ is an archimedean place, then the
assertion is the same as \cite[Lemma 3.53]{silverman_book07}. We will
not repeat the proof here.  For the case where $v$ is non-archimedean,
notice that $\l_v$ satisfies the strong triangle inequality, 
\[
 \l_v(x, y) \ge \min\left(\l_v(x,z), \l_v(y,z) \right),  
\]
and that this inequality is an equality if $\l_v(x, z) \ne \l_v(y, z)$.
Suppose that $x$ and $y$ satisfy the required condition in the statement of
the lemma, i.e., $\l_v(x,y) > \l_v(y, \infty)$. (Notice that $\ell_v = 1$
in this case.) 
We claim that $\l_v(x, \infty) \le \l_v(y, \infty)$. Assume to the
contrary that $\l_v(x, \infty) > \l_v(y, \infty)$. Then, by the strong
triangle inequality for $\l_v$ we have
\[
  \l_v(x,y) = \min\left(\l_v(x, \infty),\l_v(y, \infty)\right)
   = \l_v(y, \infty). 
\]
But this contradicts the assumption that $\l_v(x,y) > \l_v(y, \infty)$
hence the claim. Now,
\begin{align*}
  \l_v(x, y) +\log |x - y|_v  & = \l_v(x,\infty) + \l_v(y, \infty)  \\*
    & \le 2 \l_v(y,\infty) \quad \text{by the claim,}
\end{align*}
which completes the proof of the lemma.
\end{proof}

\section{Dynamics and  height functions}
\label{subsec:dynamicsheightfunction}

Let $\f : \PP^1\to \PP^1$ be a rational map on $\PP^1$ of
degree $d \ge 2$ defined over the number field $K$.  We identify $K\cup
\{\infty\} \simeq \PP^1(K)$ by fixing an affine coordinate $z$ on
$\PP^1$, so $\a\in K$ equals $[\a, 1] \in \PP^1(K)$, and the point at
infinity is $[1, 0]$. With respect to this affine coordinate, we
identity rational maps $\f:\PP^1\to\PP^1$ with rational functions
$\f(z) \in K(z)$.

Let $P\in \PP^1$. Then, the (\emph{forward}) \emph{orbit} of $P$ under
iteration of $\f$ is the set
\[
    \Orbit_\f(P) = \bigl\{\f^n(P) : n = 0, 1, 2, \ldots\bigr\}.
\]
The point $P$ is called a \emph{wandering point} of $\f$ if
$\Orbit_\f(P)$ is an infinite set; otherwise, $P$ is called a
\emph{preperiodic point} of $\f$. The set of preperiodic points of
$\f$ is denoted by $\PrePer(\f)$.  We say that a point $A\in \PP^1$ is
an \emph{exceptional point} if it is preperiodic and
$\f^{-1}\bigl(\Orbit_\f(A)\bigr)=\Orbit_\f(A)$, which is equivalent to
the assumption that the complete (forward and backward) $\f$-orbit
of~$A$ is a finite set.  It is a standard fact that $A$ is an
exceptional point for~$\f$ if and only if $A$ a totally ramified fixed
point of $\f^2$. (One direction is clear, and the other follows from
that fact~\cite[Theorem~1.6]{silverman_book07} that if~$A$ is an
exceptional point, then~$\Orbit_\f(A)$ consists of at most two points.)

For a point $P = [x_0, x_1] \in \PP^1(K)$, the
\emph{height} of~$P$ is
\[
  h(P)  = \sum_{v\in M_K}\; \frac{[K_v:\QQ_v]}{[K:\QQ]}
     \log\max\bigl(|x_0|_v, |x_1|_v\bigr).
\]
Then the \emph{canonical height} of $P$ relative to the rational map
$\f$ is given by the limit
\[  
 \hhat_{\f}(P) = \lim_{n\to \infty}\, \frac{h(\f^n P)}{d^n}.
\]
To simplify notation, we let
\[
  d_v = \frac{[K_v:\QQ_v]}{[K :\QQ]}.
\]
\par
Using the definition of $\l_v$, we see that 
\[
  h(P)=\sum_{v\in M_K} d_v\l_v(P, \infty)+O(1).
\]
More precisely, writing $P=[x_0,x_1]$ and~$\infty=[1,0]$, we have
\[
  h(P) = \sum_{v\in M_K^0} d_v\l_v(P,\infty)  
  + \sum_{v\in M_K^\infty} d_v
      \log\left(\frac{\max\bigl\{|x_0|_v,|x_1|_v\bigr\}}
                     {\sqrt{|x_0|_v^2+|x_1|_v^2}}\right).
\]
The quantity $\max\{a,b\}/\sqrt{a^2+b^2}$ is between~$1/\sqrt2$ and~$1$
for all non-negative~$a,b\in\RR$, so
\[
  -\frac12\log2 \le h(P) - \sum_{v\in M_K} d_v\l_v(P,\infty)  \le 0.
\]
For further material and basic properties of height functions, see for
example~\cite[\S\S3.1--3.5]{silverman_book07}.

For a polynomial $f = \sum a_i z^i \in K[z]$ and absolute value $v\in
M_K$, we define
\[
  |f|_v = \max\bigl\{|a_i|_v\bigr\}\quad \text{and}\quad
  h(f) = h\bigl([\ldots, a_i, \ldots]\bigr)
   =  \sum_{v\in M_K}\, d_v \log |f|_v. 
\] 
We say that a rational function~$\f(z)=f(z)/g(z)\in K(z)$ of
degree~$d$ is written in normalized form if
\[
  f(z) = \sum_{i=0}^d\, a_i z^i \quad\text{and}\quad
  g(z) = \sum_{i=0}^d\, b_i z^i
  \quad \text{with} \quad  a_i, b_i\in K,
\]
if $a_d$ and $b_d$ are not both zero, and if $f$ and $g$
are relatively prime in~$K[z]$. For $v\in M_K$, we set $|\f|_v =
\max\{|f|_v, |g|_v\}$, and then the height of $\f$ is defined by
\begin{align*}
  h(\f) & = h\bigl([a_0, \ldots, a_d, b_0, \ldots, b_d]\bigr)  \\
  & = \sum_{v\in M_K} d_v \log |\f|_v . 
\end{align*}
Directly from the definitions, we have
\begin{equation}
  \label{ineq:function-height}
   \max\bigl(h(f), h(g) \bigr)  \le h(\f). 
\end{equation}

The following basic properties of absolute values of polynomials will
be useful.

\begin{lemma}
\label{lemma:absval-poly}
Let $v\in M_K$ and let $f, g \in K[x]$ be  polynomials with
coefficients in $K$.
\begin{parts}
\Part{(a)}
\[
  |f + g|_v \le
  \smash[t]{
    \begin{cases}
      |f|_v + |g|_v & \text{if $v$ is archimedean,}  \\
      \max\{|f|_v, |g|_v\} & \text{if $v$ is nonarchimedean.}
    \end{cases}
  }
\]
\Part{(b)} \textup{(Gauss' Lemma)}
If $v$ is nonarchimedean, then  $|f g |_v = |f|_v |g|_v$.
\Part{(c)}
If $v$ is archimedean and $\deg f + \deg g < d$, then 
\[
\frac{1}{4^d} |f g|_v \le |f|_v |g|_v \le 4^d |f g|_v 
\]
\end{parts}

\begin{proof}
(a) follows from the definition.  For~(b) and~(c), see for example
\cite[Chapter~3, Propositions~2.1 and~2.3]{lang83}.
\end{proof}
\end{lemma}

\begin{proposition}
\label{prop:polynomial-height-ineq}
Let $\{f_1, \ldots, f_r\}$ be a collection of polynomials in
the ring
$K[x]$.%
\begin{parts}
\Part{(a)}
$\displaystyle
  \begin{aligned}[t]
      h(f_1f_2\cdots f_r) &\le \sum_{i=1}^r \bigl(h(f_i) + (\deg f_i +1)
      \log 2\bigr)  \\*
    & \le r \max_{1\le i \le r}\, \bigl\{h(f_i) + (\deg f_i +1)\log 2\bigr\}.
  \end{aligned}
$
\Part{(b)}
$\displaystyle
  h(f_1+f_2 + \cdots + f_r) \le \sum_{i=1}^r h(f_i) + \log r.
$
\Part{(c)}
Let $\f(z),\psi(z)\in K(z)$ be rational functions. 
Then
\[
  h(\f\circ \psi) \le h(\f) + (\deg\f )h(\psi) + (\deg\f)(\deg\psi) \log 8.
\]
\Part{(d)}
Let $\f(z)\in K(z)$ be a rational function of degree~$d\ge2$.Then for
all $n\ge 1$ we have
\[
 h(\f^n) \le \left(\frac{d^n- 1}{d-1}\right) h(\f)  +
 d^2\left(\frac{d^{n-1}- 1}{d-1}\right) \log 8. 
\]
\end{parts}
\end{proposition}

\begin{proof}
The proofs of (a) and (b) can be found in
\cite[Proposition~B.7.2]{hindry-silverman00}, where the proposition is
stated for multi-variable polynomials. As we'll use the arguments in
(a) for the proof of (c), we repeat the proof of~(a) for the
one-variable case. (Also, our situation is slightly different from
\cite{hindry-silverman00}, since we are using a projective height,
while \cite{hindry-silverman00} uses an affine height.)  Writing $f_i
= \sum_E a_{iE} X^E$, we have
\[
  f_1 \cdots f_r = \sum_E \left(\sum_{e_1+\cdots +e_r = E}\,
  a_{1e_1}\cdots a_{r e_r} \right) X^E,
\]
and hence for $v \in M_K$,
\begin{equation}
  \label{eqn:f1frmaxE}
  |f_1\cdots f_r|_v 
  = \max_E \left|\sum_{e_1+\cdots +e_r =  E}
      \,a_{1e_1}\cdots a_{r e_r}\right|_v 
\end{equation}
and
\[
  h(f_1\cdots f_r) = \sum_{v\in M_K} \,d_v \log |f_1\cdots f_r|_v.
\]
If $v$ is nonarchimedean, then by Gauss' Lemma
(Lemma~\ref{lemma:absval-poly}(b)) we have
\[
  |f_1\cdots f_r|_v = \prod_{i=1}^r \, |f_i|_v .
\]
\par
It remains to deal with archimedean place $v$.  We note that the
number of terms in the sum appearing in the right-hand side
of~\eqref{eqn:f1frmaxE} is $\binom{E+r - 1}{E}$.
Hence
\begin{align*}
  |f_1\cdots f_r|_v 
  &\le  \max_E \left(\binom{E+r-1}{E} \max_{e_1+\cdots +e_r =E}
    |a_{1e_1}\cdots a_{r e_r}|_v \right) \\
  &\le  \max_E \left(2^{E+r-1} \max_{e_1+\cdots +e_r =E}
    |a_{1e_1}\cdots a_{r e_r}|_v \right).
\end{align*}
Further, if $E>\deg(f_1\dots f_r)$, then the product $a_{1e_1}\cdots
a_{r e_r}$ is zero , since in that case at least one of the $a_{ij}$ is
zero. Hence
\begin{equation}
  \label{ineq:archim-ineq}
  |f_1\cdots f_r|_v 
  \le  2^{\deg(f_1\cdots f_r) + r - 1} \prod_{i=1}^r \, |f_i|_v .
\end{equation}
Let $N_v = 2^{\sum_i(\deg f_i + 1)}$ if $v$ is archimedean, and $N_v =
1$ if $v$ is non-archimedean. Then we compute
\begin{align*}
  h(f_1 \cdots f_r) & = \sum_{v\in M_K}\,d_v\log |f_1\cdots f_r|_v  \\
  & \le \sum_{v\in M_K}\, d_v \left(\log N_v + \log \prod_{i=1}^r
   |f_i|_v\right)  \\
  & \le \sum_{i=1}^r \left(h(f_i) + (\deg f_i + 1)\log 2\right)  \\
  & \le r \max_{1\le i \le r}\, \{h(f_i) + (\deg f_i + 1)\log 2\},
\end{align*}
which completes the proof of~(a).
\par
Next we give a proof of (c).  Write $\psi = \psi_0/\psi_1 \in K(z)$
in normalized form, so in particular $\psi_0$ and $\psi_1$ are
relatively prime polynomials. Then
\[
  (\f\circ\psi) (z)
   = \frac{\sum a_i \psi_0^i \psi_1^{d-i}}{\sum b_i \psi_0^i \psi_1^{d-i}},
\]
so by definition of the height of a rational function we have 
\[
  h(\f\circ\psi) 
  \le \sum_{v\in M_K}\, d_v 
    \log \max\left\{\left|\sum a_i \psi_0^i\psi_1^{d-i}\right|_v, 
         \left|\sum b_i \psi_0^i \psi_1^{d-i}\right|_v\right\}. 
\]
For the right hand side of the above inequality, 
if $v$ is nonarchimedean, then by  Gauss' Lemma again we have 
\[
  \left|\sum a_i \psi_0^i\psi_1^{d-i}\right|_v 
   \le \max\bigl(|f|_v |\psi_0|_v^i |\psi_1|_v^{d-i}\bigr) 
   \le |\f|_v |\psi|_v^d. 
\]
Similarly,
\[
  \left|\sum b_i \psi_0^i\psi^{d-i}\right|_v \le |\f|_v |\psi|_v^d . 
\]
Hence for $v$ nonarchimedean,
\[
|\f\circ\psi|_v \le |\f|_v |\psi|_v^d . 
\]
\par
Next let $v$ be an archimedean place of $K$. Then
the triangle inequality gives
\[
  \left|\sum a_i \psi_0^i\psi_1^{d-i}\right|_v
 \le (d+1) |f|_v \max_i \bigl\{|\psi_0^i\psi_1^{d-i}|_v\bigr \}.
\]
Applying the estimate (\ref{ineq:archim-ineq}) to the product
$\psi_0^i \psi_1^{d-i}$ yields
\[
  |\psi_0^i \psi_1^{d-i}|_v 
  \le 2^{d(\deg\psi + 1)} |\psi_0|_v^i |\psi_1|_v^{d-i} 
  \le 2^{d(\deg\psi + 1)} |\psi|_v^d.
\]
Therefore,
\[
  \left|\sum a_i \psi_0^i\psi_1^{d-1}\right|_v   
  \le (d+1)2^{d(\deg\psi + 1)}  |f|_v  |\psi|_v^d
  \le (d+1)2^{d(\deg\psi + 1)}  |\f|_v  |\psi|_v^d. 
\]
\par
Similarly,
\[
  \left|\sum b_i \psi_0^i\psi_1^{d-1}\right|_v   
  \le (d+1)2^{d(\deg\psi + 1)}  |\f|_v  |\psi|_v^d. 
\]
\par  
We combine these estimates. To ease notation, we let $N_v = 1$ for $v$
nonarchimedean and $N_v = (d+1)2^{2d\deg\psi} = (d+1)4^{\deg\f\deg\psi}$
for $v$ archimedean. Then
\begin{align*}
  h(\f\circ \psi)  
  & \le  \sum_{v\in M_K}\, d_v \log \max\left\{\left|\sum a_i
  \psi_0^i  \psi_1^{d-1}\right|_v, 
   \left|\sum b_i \psi_0^i \psi_1^{d-1}\right|_v\right\}
  \\
  & \le \sum_{v\in M_K}\, d_v 
    \bigl(\log |\f|_v  + d \log |\psi|_v + \log N_v\bigr)
  \\
  & \le h(\f) + d h(\psi) + (\deg\f) (\deg\psi) \log 4  + \log (d+1)
  \\
  & \le  h(\f) + d h(\psi) +  (\deg\f) (\deg\psi) \log 8,
\end{align*}
since $d + 1 \le 2^d \le 2^{d \deg\psi}$.
This completes the proof of (c). 
\par
Finally, we prove~(d) by induction on~$n$. The stated inequality is clearly
true for~$n=1$. Assume now it it true for~$n$. Then
\begin{align*}
  h(\f^{n+1})
  &\le h(\f^n) + d^nh(\f) + d^{n+1}\log 8
    &&\text{from (c) applied to $\f^n$ and $\f$,} \\
  &\le \smash[b]{\left(\frac{d^n- 1}{d-1}h(\f)+d^2\frac{d^{n-1}-
        1}{d-1}\log 8\right)} 
        + d^n h(\f) + d^{n+1}\log 8 \hidewidth\\
  &  &&\text{from the induction hypothesis,} \\
  &= \left(\frac{d^{n+1}- 1}{d-1}\right)h(\f) + d^2\left(\frac{d^n-
      1}{d-1}\right)\log 8.\hidewidth
\end{align*}
This completes the proof of Proposition~\ref{prop:polynomial-height-ineq}.
\end{proof}

The following facts about height functions are well-known.

\begin{proposition}
\label{prop:heightproperty}
Let $\f : \PP^1 \to \PP^1$ be a rational map of degree $d\ge2$ defined
over $K$.  There are constants $c_1$,~$c_2$,~$c_3$, and $c_4$,
depending only on $d$, such that the following estimates hold for all
$P \in \PP^1(\Kbar)$.
\begin{parts}
\Part{(a)} 
$\bigl|h(\f(P)) - d h(P)\bigr| \le c_1 h(\f) + c_2$.
\Part{(b)}
$ \bigl| \hhat_\f(P) - h(P)\bigr| \le c_3 h(\f) + c_4$.
\Part{(c)}
$\hhat_\f(\f(P)) = d \hhat_\f(P)$.
\Part{(d)}
$P \in \PrePer(\f)$ if and only if $\hhat_\f(P) = 0$.
\end{parts}
\end{proposition}
\begin{proof}
See, for example, \cite[\S\S B.2,B.4]{hindry-silverman00}
or \cite[\S3.4]{silverman_book07}.  
\end{proof}

\section{A Distance Estimate}
\label{subsec:dist}

Our goal in this section is a version of the inverse function theorem
that gives explicit estimates for the dependence on the (local) heights of
both the points and the function. 
It is undoubtedly possible to give a direct, albeit long and messy,
proof of the desired result. We instead give a proof using universal
families of maps and arithmetic distance functions. Before stating our
result, we set notation for the universal family of degree~$d$
rational maps on~$\PP^1$.
\par
We write~$\Rat_d\subset\PP^{2d+1}$ for the space of rational maps of
degree~$d$, where we identify a rational map $\f=f/g=\sum
a_iz^i\bigm/\sum b_iz^i$ with the point
\[
  [\f]=[f,g]=[a_0,\ldots,a_d,b_0,\ldots,b_d]\in\PP^{2d+1}.
\]
If $\f\in\Rat_d(\Qbar)$ is defined over~$\Qbar$, we define the height
of~$\f$ as in Section~\ref{subsec:dynamicsheightfunction}
to be the height of the associated point in~$\PP^{2d+1}(\Qbar)$,
\[
  h(\f)=h\bigl([a_0,\ldots,a_d,b_0,\ldots,b_d]\bigr).
\]
\par
Over~$\Rat_d$, there is a universal family of degree~$d$ maps, which we
denote by
\[
  \Psi : \PP^1\times\Rat_d\longrightarrow\PP^1\times\Rat_d,\qquad
  (P,\psi)\longmapsto \bigl(\psi(P),\psi\bigr).
\]
We note that~$\Rat_d$ is the complement in~$\PP^{2d+1}$ of a
hypersurface, which we denote by~$\partial\Rat_d$. (The
set~$\partial\Rat_d$ is given by the resultant $\Resultant(f,g)=0$,
so~$\partial\Rat_d$ is a hypersurface of degree~$2d$.)
Since~$\PP^1$ is complete, we have
\[
  \partial(\PP^1\times\Rat_d) = \PP^1\times\partial\Rat_d.
\]
\par
The map~$\Psi$ is a finite map of degree~$d$.  Let~$R(\Psi)$ denote
its ramification locus. Looking at the behavior of~$\Psi$ in a
neighborhood of a point~$(P,\psi)$, it is easy to see that the
restriction of~$R(\Psi)$ to a fiber~$\PP^1_\psi=\PP^1\times\{\psi\}$
is the ramification divisor of~$\psi$,
\[
  R(\Psi)\big|_{\PP^1_\psi} = R(\psi).
\]
So the ramification indices of the universal map~$\Psi$ and a particular
map~$\psi$ are related by
\begin{equation}
  \label{eqn:ePsi}
  e_{(P,\psi)}(\Psi)=e_P(\psi).
\end{equation}

\begin{proposition}
\label{prop:inv-fun}
Let $\psi \in K(z)$ be a nontrivial rational function, let
\text{$S\subset M_K$} be a finite set of absolute values on~$K$, each
extended in some way to~$\Kbar$, and let \text{$A,P\in\PP^1(K)$}. Then
\begin{multline*}
  \smash[b]{\sum_{v\in S}  \max_{A'\in\psi^{-1}(A)} e_{A'}(\psi)d_v\l_v(P,A')}\\
  \ge \sum_{v\in S} d_v\l_v\bigl(\psi(P),A\bigr)
     + O\bigl(h(A)+h(\psi)+1\bigr),
\end{multline*}
where the implied constant depends only on the degree of the
map~$\psi$.
\end{proposition}
\begin{proof}
The statement and proof of Proposition~\ref{prop:inv-fun} use the
machinery of arithmetic distance functions and local height functions
on quasi-projective varieties as described
in~\cite{silverman_mathann87}, to which we refer the reader for
definitions, notation, and basic properties. We begin with the
distribution relation for finite maps of smooth quasi-projective
varieties~\cite[Proposition~6.2(b)]{silverman_mathann87}. Applying
this relation to the map~$\Psi$ and points~$x,y\in\PP^1\times\Rat_d$
yields
\begin{equation}
  \label{eq:distrel}
  \d\bigl(\Psi(x),y;v\bigr)
  = \sum_{y'\in\Psi^{-1}(y)} e_{y'}(\Psi)\d(x,y';v) 
  + O\left( \l_{\partial(\PP^1\times\Rat_d)^2}(x,y;v) \right).
\end{equation}
Here~$\d(\,\cdot\,,\,\cdot\,;v)$ is a $v$-adic arithmetic distance
function on~$\PP^1\times\Rat_d$ and~$\l_{\partial(\PP^1\times\Rat_d)^2}$
is a local height function for the indicated divisor.  In particular,
if we take $x=(P,\psi)$ and $y=(A,\psi)$, then the arithmetic distance
function~$\d$ and the chordal metric~$\l_v$ defined in
Section~\ref{section:preliminarynotation} satisfy
\begin{align} 
  \label{eqn:dPsiyv}
  \d\bigl(\Psi(x),y;v\bigr)
  = \d\bigl(\Psi(P,\psi),(A,\psi);v\bigr)
  &= \d\bigl((\psi(P),\psi),(A,\psi);v\bigr)\notag\\
  &= \l_v\bigl(\psi(P),A\bigr).
\end{align}
Similarly, if~$y'=(A',\psi)\in\Psi^{-1}(y)$, then
\[
  \d(x,y';v)
  = \d\bigl((P,\psi),(A',\psi);v\bigr)
  = \l_v(P,A').
\]
Further, since $\partial(\PP^1\times\Rat_d) = \PP^1\times\partial\Rat_d$
is the pull-back of a divisor on~$\Rat_d$ and
\[
  \partial(\PP^1\times\Rat_d)^2
  = (\PP^1\times\partial\Rat_d)\times(\PP^1\times\Rat_d)
    + (\PP^1\times\Rat_d)\times(\PP^1\times\partial\Rat_d),
\]
applying
\cite[Proposition~5.3~(a)]{silverman_mathann87} gives
\begin{align}
  \label{eqn:lP1Ratd}
  \l_{\partial(\PP^1\times\Rat_d)^2}(x,y;v)  
  &\gg\ll \l_{\PP^1\times\partial\Rat_d}\bigl((P,\psi);v\bigr)
    + \l_{\PP^1\times\partial\Rat_d}\bigl((A,\psi);v\bigr) \notag\\*
  &\gg\ll \l_{\partial\Rat_d}(\psi;v).
\end{align}
\par
Substituting~\eqref{eqn:ePsi},~\eqref{eqn:dPsiyv},
and~\eqref{eqn:lP1Ratd} into the distribution
relation~\eqref{eq:distrel} yields
\begin{equation}
  \label{eqn;lvpisPA}
  \l_v\bigl(\psi(P),A\bigr)
  = \sum_{A'\in\psi^{-1}(A)} e_{A'}(\psi)\l_v(P,A')
   + O\bigl(\l_{\partial\Rat_d}(\psi;v)\bigr).
\end{equation}
\par
To ease notation, let~$A_v'\in\psi^{-1}(A)$ be a point satisfying
\[
  e_{A_v'}(\psi) \l_v(P,A_v')
  = \max_{A'\in\psi^{-1}(A)} e_{A'}\l_v(P,A').
\]
Then for any $A'\in\psi^{-1}(A)$ we have
\begin{align}
  \label{eqn:eApsilvPA}
  e_{A'}(\psi) \l_v(P,A')
  &=\min\bigl\{ e_{A_v'}(\psi)\l_v(P,A_v'),e_{A'}(\psi) \l_v(P,A') \bigr\} 
     \notag\\*
  &\omit\hfill\text{from choice of $A_v'$,} \notag\\
  &\le d\min\bigl\{\l_v(P,A_v'), \l_v(P,A') \bigr\}
  \quad\text{since $\psi$ has degree $d$,} \notag\\
  &\le d \l_v(A_v',A')+O(1)
  \quad\text{from the triangle inequality.}
\end{align}
This is a nontrivial estimate for~$A'\ne A_v'$, so
in~\eqref{eqn;lvpisPA} we pull off the~$A_v'$ term and
use~\eqref{eqn:eApsilvPA} for the other terms to obtain
\begin{equation}
  \label{eqn:lvpsiPAe}
  \l_v\bigl(\psi(P),A\bigr)
  \le e_{A_v'}(\psi)\l_v(P,A_v')
    +  d\sum_{\substack{\hidewidth A'\in\psi^{-1}(A)\hidewidth \\A'\ne A_v'}}
           \l_v(A_v',A')
    + O\bigl(\l_{\partial\Rat_d}(\psi;v)\bigr).
\end{equation}
\par
The next lemma gives an upper bound for~$\l_v(A_v',A')$.

\begin{lemma}
\label{lemma:distinvimage}
There is a constant~$C=C(d)$ such that the following holds.  Let
$\psi\in\Rat_d(\Qbar)$, let~$A\in\PP^1(\Qbar)$, and
let~$A',A''\in\psi^{-1}(A)$ be \emph{distinct} points. Then
\[
  \sum_{v\in M_K}   d_v\l_v(A',A'')
  \le C \bigl(h(A)+h(\psi)+1\bigr).
\]
\end{lemma}
\begin{proof}
In the notation of~\cite{silverman_mathann87}, we have
\begin{align*}
  \l_v(A',A'') &= \d_{\PP^1\times\Rat_d}\bigl((A',\psi),(A'',\psi);v\bigr)\\
  &= \l_{(\PP^1\times\Rat_d)^2,\D}\bigl((A',\psi),(A'',\psi);v\bigr),
\end{align*}
where~$\D$ is the diagonal of $(\PP^1\times\Rat_d)^2$. 
Summing over~$v$ gives height functions
\begin{align*}
  \smash[b]{
    \sum_{v\in M_K}\l_v(A',A'')
  }
  &= h_{(\PP^1\times\Rat_d)^2,\D}\bigl((A',\psi),(A'',\psi)\bigr) \\
  &\qquad{}
  + O\bigl(h_{\partial(\PP^1\times\Rat_d)^2}\bigl((A',\psi),(A'',\psi)\bigr)+1
   \bigr).
\end{align*}
Choosing an ample divisor~$H$ on~$\PP^1\times\Rat_d$, we use the fact
that heights with respect to a subscheme are dominated by ample
heights away from the support of the
subscheme~\cite[Proposition~4.2]{silverman_mathann87}. (This is where
we use the assumption that~$A'\ne A''$, which ensures that the 
point~$\bigl((A',\psi),(A'',\psi)\bigr)$ is not on the diagonal.) This yields
\begin{align}
  \label{eqn:lvAhAA}
  \smash[b]{
    \sum_{v\in M_K}\l_v(A',A'')
  }
  &\ll h_{\PP^1\times\Rat_d,H}(A',\psi)+ h_{\PP^1\times\Rat_d,H}(A'',\psi)+1
    \notag\\
  &\ll h(A')+h(A'')+h(\psi)+1.
\end{align}
\par
We now use~\cite[Theorem~2]{silverman_preprint09}, which says that
there are positive constants~$C_1,C_2,C_3$, depending only on the
degree of~$\psi$, such that
\begin{equation}
  \label{eqn:hpsiPChP}
  h\bigl(\psi(P)\bigr)\ge C_1h(P)-C_2h(\psi)-C_3.
\end{equation}
(The paper~\cite{silverman_preprint09} deals with general rational
maps~$\PP^n\dashrightarrow\PP^n$. In our case with~$n=1$, it would be
a tedious, but not difficult, calculation to give explicit values for
the~$C_i$, including of course~$C_1=\deg\psi$.) Applying~\eqref{eqn:hpsiPChP}
with~$P=A'$ and~$P=A''$, we substitute into~\eqref{eqn:lvAhAA} to obtain
\[
  \sum_{v\in M_K}\l_v(A',A'')
  \ll h(A) + h(\psi)+1,
\]
which completes the proof of Lemma~\ref{lemma:distinvimage}.
\end{proof}

We use Lemma~\ref{lemma:distinvimage} to bound the sum in 
the right-hand side of the inequality~\eqref{eqn:lvpsiPAe}.
We note that~$\l_v(A',A'')\ge0$ for all points, so the lemma
implies in particular that~$\sum_{v\in S}d_v\l_v(A',A'')\ll
h(A)+h(\psi)+1$ for any set of places~$S$.
Further, the sum in~\eqref{eqn:lvpsiPAe} has at most~$d-1$ terms.
Hence we obtain 
\[
  \sum_{v\in S} d_v\l_v\bigl(\psi(P),A\bigr)  
  \le \sum_{v\in S} e_{A_v'}(\psi)d_v\l_v(P,A_v')
    + O\bigl(h(A)+h(\psi)+1\bigr).
\]
Note that in this last inequality, the $O\bigl(h(\psi)\bigr)$ term
comes from two places,  Lemma~\ref{lemma:distinvimage} and 
\[
  \sum_{v\in S}d_v \l_{\partial\Rat_d}(\psi;v)
  \le \sum_{v\in M_K}d_v \l_{\partial\Rat_d}(\psi;v)
  = h_{\partial\Rat_d}(\psi) = O\bigl(h(\psi)+1\bigr),
\]
where the last equality comes from the fact that $\partial\Rat_d$ is
a hypersurface of degree $2 d$ in $\PP^{2d+1}.$   
This completes the proof of Proposition~\ref{prop:inv-fun}.
\end{proof}

\section{A Ramification Estimate and a Quantitative Version of Roth's
  Theorem} 
\label{subsec:distandroth}
In this section we state two known results that will be needed to
prove our main theorem. The first says that away from exceptional
points, the ramification of~$\f^m$ tends to spread out as~$m$
increases.

\begin{lemma}
\label{lemma:ramifindex}
Fix an integer~$d\ge2$.  There exist constants $\kappa_1$ and
$\kappa_2<1$, depending only on $d$, such that for all degree~$d$
rational maps $\f : \PP^1 \to \PP^1$, all points $Q \in \PP^1$ that
are not exceptional for~$\f$, all integers $m \ge 1$, and all $ P \in
\f^{-m}(Q)$, we have
\[
  e_P(\f^m) \le \kappa_1 (\kappa_2 d)^m.  
\]
\end{lemma}
\begin{proof}
This is \cite[Lemma~3.52]{silverman_book07}; see in particular the last
paragraph of the proof. It is not difficult to give explicit values
for the constants. In particular, if~$Q$ is not preperiodic, then
the stronger estimate \text{$e_P(\f^m)\le e^{2d-2}$} is true for all~$m$.
\end{proof}

The second result we need is the following quantitative version of
Roth's Theorem.

\begin{theorem}
\label{thm:roth}
Let $S$ be a finite subset of $M_K$ that contains all infinite
places. We assume that each place in $S$ is extended to $\Kbar$ in some
fashion. Set the following notation.
\begin{notation}
\item[$s$]
  the cardinality of $S$.
\item[$\Upsilon$]
  a finite, $G_{\Kbar/K}$-invariant subset of $K$.
\item[$\b$]
  a map $ S \to \Upsilon$.
\item[$\mu > 2$]
  a constant.
\item[$M\ge 0$]
  a constant.
\end{notation}
There are constants $r_1$ and $r_2$, depending only on $[K :\QQ]$,
$\#\Upsilon$, and $\mu$, such that there are at most $4^{s}r_1$
elements $x\in K$ satisfying both of the following conditions\textup:
\begin{gather}
  \label{roth1}
  \sum_{v\in S}\, d_v \log^+|x - \b_v|_v^{-1}   \ge \mu h(x) -  M.
  \\*
  \label{roth2}
  \displaystyle h(x)  \ge r_2 \max_{v\in S}\{h(\b_v), M, 1\}.
\end{gather}
\end{theorem}
\begin{proof}
This is \cite[Theorem~2.1]{silverman_crelles87}, with a small change
of notation.  For explicit values of the constants,
see~\cite{rob.gross90}.
\end{proof}  

\section{A Bound for the Number of Quasi-Integral Points in an Orbit}
\label{sec:main-theorem}
In this section we prove our main result, which is an explicit upper
bound for the number of iterates~$\f^n(P)$ that are close to a given
base point~$A$ in any one of a fixed finite number of $v$-adic
topologies. Here is the precise statement.

\begin{theorem}
\label{thm:main-theorem}
Let $\f\in K(z)$ be a rational map of degree $d \ge 2$.  Fix a point
$A \in \PP^1(K)$ which is not an exceptional point for $\f$, and let
$P\in \PP^1(K)$ be a wandering point for $\f$. For any finite set of
places $S\subset M_K$ and any constant $1\ge \ve > 0$, define a set of
non-negative integers
\[
  \G_{\f,S}(A, P, \ve)
  = \left\{n\ge0 : \sum_{v\in S}  d_v \l_v(\f^n P, A) 
  \ge  \ve \hhat_{\f}(\f^n P) \right\}.
\]
\begin{parts}
\Part{(a)}
There exist constants 
\[
  \g_1=\g_1\bigl(d,\ve,[K:\QQ]\bigr)
  \quad\text{and}\quad
  \g_2=\g_2\bigl(d,\ve,[K:\QQ]\bigr)
\]
such that
\begin{equation}
  \label{eqn:GfSApebound1}
  \#\left\{ n\in \G_{\f,S}(A, P, \ve)  :
      n > \g_1 + \log^+_d\left(\frac{h(\f) +  \hhat_\f(A)}{\hhat_\f(P)}\right)
  \right\}
  \le 4^{\#S}\g_2.
\end{equation}
\Part{(b)}
In particular, there is a constant $\g_3=\g_3\bigl(d,\ve,[K:\QQ]\bigr)$
such that
\begin{equation}
  \label{eqn:GfSApebound}
  \# \G_{\f,S}(A, P, \ve)  
  \le  4^{\#S}\g_3
      + \log^+_d\left(\frac{h(\f) +  \hhat_\f(A)}{\hhat_\f(P)}\right).
\end{equation}
\Part{(c)}
There is a constant~$\g_4=\g_4\bigl(K,S,\f,A,\e)$ that is independent
of~$P$ such that
\[
  \max \G_{\f,S}(A, P, \ve)  \le \g_4.
\]
\end{parts}
\end{theorem}

Before giving the proof of Theorem~\ref{thm:main-theorem}, we make a
number of remarks.

\begin{remark}
Note that as a consequence of
Proposition~\ref{prop:heightproperty}(d), we have $\hhat_\f(P) >
0$ if $P$ is wandering point for $\f.$ Hence the right-hand side of
(\ref{eqn:GfSApebound}) is well defined. 
\end{remark}

\begin{remark}
If we take $\ve=1$, then the set~$\G_{\f,S}(A,P,\ve)$ more-or-less
coincides with the set of points in the orbit~$\Orbit_\f(P)$ that are
$S$-integral with respect to~$A$. We say more-or-less
because~$\G_{\f,S}(A,P,\ve)$ is defined using the canonical height
of~$\f^n(P)$, rather than the naive height. But using the inequality
$\bigl| \hhat_\f(P) - h(P)\bigr| \ll h(\f) + 1$ from
Proposition~\ref{prop:heightproperty} and adjusting the constants, it
is not hard to see that the estimate~\eqref{eqn:GfSApebound} remains
true for the set
\[
  \G^{\text{naive}}_{\f,S}(A, P, \ve)
  = \left\{n\ge0 : \sum_{v\in S}  d_v \l_v(\f^n P, A) 
  \ge  \ve h(\f^n P) \right\}.
\]
(See the proof of Corollary~\ref{cor:s-integer}.)  For example,
taking~$A=\infty$, the set $\G^{\text{naive}}_{\f,S}(A, P, \ve)$
consists of the points~$\f^n(P)$ such that $z\bigl(\f^n(P)\bigr)$ is
$(S, \ve_0)$-integral for some $\ve_0$. This is the motivation for
saying that the points in~$\G_{\f,S}(A,P,\ve)$ are
quasi-$(S,\ve)$-integral with respect to~$A$, where~$\ve$ measures the
degree of~$S$-integrality.
\end{remark}

\begin{remark}
\label{rmk:dynamical_langconj}  
The dependence of the bounds~\eqref{eqn:GfSApebound1}
and~\eqref{eqn:GfSApebound}
on~$h(\f)$, $\hhat_\f(A)$, and $\hhat_v(P)$ are quite interesting.  A
dynamical analogue of a conjecture of Lang asserts that the
ratio~$h(\f)/\hhat_\f(P)$ is bounded, independently of~$\f$ and~$P$,
provided that~$\f$ is suitably minimal with respect
to~$\PGL_2(K)$-conjugation. See~\cite[Conjecture~4.98]{silverman_book07}.
\par
On the other hand, there cannot be a uniform bound for the ratio
$\hhat_\f(A)/\hhat_\f(P)$, since~$A$ and~$P$ may be chosen arbitrarily
and independent of one another. This raises the interesting question
of whether the bound for $\#\G_{\f,S}(A, P, \ve)$ actually needs to
depend on~$A$. Even in very simple situations, it appears difficult to
answer this question. For example, consider the map~$\f(z)=z^2$, the
initial point~$P=2$, and the set of primes~$S=\{\infty,3,5\}$.  
As~$A\in\QQ^*$ varies, is it possible for the
orbit~$\Orbit_\f(P)$ to contain more and more points that
are~$S$-integral with respect to~$A$?  Writing $A=x/y$, we are asking
if
\[
  \sup_{x,y\in \ZZ} 
  \#\bigl\{ (n,i,j)\in\NN^3 :   y\cdot 2^{2^n} - x = 3^i5^j\bigr\}
  = \infty.
\]
\end{remark}

\begin{remark}
We observe that $\#\G_{\f,S}(A, P, \ve)$ can grow as fast
as~$\log(\ve^{-1})$ as~$\ve\to0^+$.  For example, consider the
map~$\f(z)=z^d+z^{d-1}$, the points~$A=0$ and $P=p$, and the set of
primes $S=\{p\}$. Since $\f^n(z)=z^{(d-1)^n}+\text{h.o.t.}$, we have
$\bigl|\f^n(p)\bigr|_p=p^{-(d-1)^n}$, so
\[
  \l_p(\f^nP,A)=\l_p(\f^n(p),0)
  = -\log\bigl|\f^n(p)\bigr|_p=(d-1)^n\log p.
\]
Thus $\G_{\f,S}(A, P, \ve)$ consists of all $n\ge0$ satisfying
\[
  (d-1)^n\log p \ge \ve\hhat_\f(\f^n P) = \ve d^n\hhat_\f(P).
\]
Hence
\begin{align*}
  \#\G_{\f,S}(A, P, \ve)
  &= \left\lfloor\left.{\log\left(\dfrac{\log p}{\ve\hhat_\f(P)}\right)}\right/
                {\log\left(\dfrac{d}{d-1}\right)}\right\rfloor + 1 \\
  &= \frac{\log(\ve^{-1})}{\log (d/(d-1))} + o(\log\ve^{-1})
    \qquad\text{as $\e\to0^+$.}  
\end{align*}
In particular, if~$\ve$ is small and~$d$ is large, so
$\log( d/(d-1))\approx1/(d-1)$, then we have
\[
  \#\G_{\f,S}(A, P, \ve) \approx (d-1)\log(\e^{-1}).
\]
\end{remark}

\begin{remark}
See~\cite{gross-silverman95,silverman_crelles87} for a version of
Theorem~\ref{thm:main-theorem} for elliptic curves. These papers
deal with points on an elliptic curve~$E$ that are quasi-$(S,\e)$-integral with
respect to~$O$, the zero point of~$E$. It is also of interest to
study points that are integral with respect to some other point~$A$,
and in particular to see how the bound depends on~$A$. The distance
function on~$E$ is translation invariant up to~$O\bigl(h(E)\bigr)$,
so we want to estimate the size of the set
\begin{equation}
  \label{eqn:intptsonE}
  \biggl\{P \in E(K) : \sum_{v\in S} d_v \l_v(P - A) \ge \ve
  \hhat_E(P) \biggr\}.
\end{equation}
Translating the points in~\eqref{eqn:intptsonE}
by~$A$, we want to count points satisfying $\sum
d_v\l_v(P)\ge\hhat_E(P+A)+O\bigl(h(E)\bigr)$. The canonical height on
an elliptic curve is a quadratic form, so
$\hhat_E(P+A)\le2\hhat_E(P)+2\hhat_E(A)$. Using the results
in~\cite{silverman_crelles87}, this leads to a bound for the
set~\eqref{eqn:intptsonE} in which the dependence on~$A$ appears as
the ratio~$\hhat_E(A)/\hhat_E(P_{\min})$, where~$P_{\min}$ is the
point of smallest nonzero height in~$E(K)$.  This is analogous to the
dependence on~$A$ in~\eqref{eqn:GfSApebound}.
\end{remark}

\begin{proof}[Proof of Theorem~$\ref{thm:main-theorem}$]
To ease notation, we will write $\G_S(\ve)$ in place of $\G_{\f,S}(A,
P, \ve)$.  For the given $\ve$, we set $m\ge1$ to be the smallest
integer satisfying
\[
  \kappa_2^m \le \frac{\ve}{5 \kappa_1},
\]
where $\kappa_1$ and $\kappa_2$ are the positive constants appearing in
Lemma~\ref{lemma:ramifindex}.  Since $\kappa_2 < 1$, there exists such
an integer $m$.  Notice that $\kappa_1$ and $\kappa_2$ depend only on
$d$, and consequently $m$ depends only on $d$ and $\ve$.
More precisely, if we assume (without loss of generality) that $\ve<\frac12$,
then~$m\ll\log(\ve^{-1})$, where the implied constant depends only on~$d$.
\par
Put
\[
  \bfe_m
   = \max_{A'\in \f^{-m} (A)}\, e_{A'}(\f^m).
\]
Then Lemma~\ref{lemma:ramifindex} and our choice of~$m$ imply that
\begin{equation}
  \label{eqn:emk1k2dm}
  \bfe_m  \le \kappa_1(\kappa_2 d)^m  \le \frac{\ve}{5}d^m.
\end{equation}
Further, Proposition~\ref{prop:inv-fun} says that for all $Q \in
\PP^1(K)$ we have
\begin{align}
  \label{eqn:ramf-appr}
  \bfe_m \sum_{v\in S} \max_{A'\in \f^{-m}(A)} d_v &\l_v(Q, A') \notag\\* 
  &\ge \sum_{v\in S} d_v\l_v(\f^m Q, A) - O\bigl(h(A)+h(\f^m)+1\bigr),
\end{align}
where the implied constant depends on~$\deg(\f^m)$.
\par
Suppose first that $n \le m$ for all $ n \in \G_S(\ve)$. Then clearly
$\#\G_S(\ve) \le m$, and from our choice of~$m$
we have
\[
  \#\G_S(\ve) \le m 
   \le \frac{\log(5\kappa_1)+\log(\ve^{-1})}{\log(\kappa_2^{-1})}+1.
\]
This upper bound has the desired form, since~$\kappa_1>0$
and~$1>\kappa_2>0$ depend only on~$d$.

We may thus assume that there exists an $ n \in \G_S(\ve)$ such that
$n>m$, and we fix such an $n \in \G_S(\ve).$  
By the definition of $\G_{S}(\ve)$ we have
\[
   \ve \hhat_{\f}(\f^n P) \le \sum_{v\in S} d_v\, \l_v(\f^n P, A).
\]
Applying~\eqref{eqn:ramf-appr} to the point $Q=\f^{n-m}(P)$ yields
\begin{align}
\label{eqn:bound1}
  \ve \hhat_{\f}(\f^n P) \le \bfe_m
  \smash[b]{\sum_{v\in S}}
  \,d_v  \max_{A'\in
  \f^{-m}(A)}\, &\l_v(\f^{n-m}P, A')  \notag\\*
  &{} + O\bigl(h(A)+h(\f^m)+1\bigr),
\end{align}
where the big-$O$ constant depends on~$\deg\f^m=d^m$, so on~$d$
and~$\ve$.

For each $v\in S$ we choose an $A'_v \in \f^{-m}(A)$ satisfying
\[
 \l_v(\f^{n-m} P, A_v') = \max_{A'\in \f^{-m}A}\l_v(\f^{n-m}P, A').
\]
(For ease of exposition, we will assume that $z(A') \ne \infty$ for
all $A' \in \f^{-m} A$. If this is not the case, then we use~$z$
for some of the~$A'$'s, and we use~$z^{-1}$ for the others.)

Let~$S'\subset S$ be the set of places in~$S$ defined by
\[
  S' = \bigl\{v\in S : \l_v\bigl(\f^{n-m}(P),A_v'\bigr)
     >  \l_v(A_v',\infty) + \log\ell_v\bigr\},
\]
where we recall that $\ell_v=2$ if~$v$ is archimedean and~$\ell_v=1$
otherwise.  Set $S'' = S\setminus S'$.  Applying
Lemma~\ref{lemma:log_chordal} to the places in~$S'$ and using the
definition of~$S''$ for the places in~$S''$, we find that
\begin{align*}
  \ve &\hhat_{\f}(\f^n(P)) \\*
  &\le \biggl(\sum_{v\in S'} + \sum_{v\in S''}\biggr) d_v\l_v(\f^nP,A) 
    \qquad\hfill\text{since $n\in\G_S(A,P,\ve)$,} \\
  &\le \bfe_m \biggl(\sum_{v\in S'} + \sum_{v\in S''}\biggr) 
     d_v \l_v \bigl(\f^{n-m}(P), A_v'\bigr)  + O\bigl(h(A)+h(\f^m)+1\bigr) \\*
    &\omit\hfill\text{from the definition of $A'_v$ and \eqref{eqn:bound1},} \\
  &\le \bfe_m\sum_{v\in S'} d_v\bigl(2\l_v(A_v',\infty)
       -\log\bigl|z\bigl(\f^{n-m}(P)\bigr)-z(A_v')\bigr| +\log\ell_v\bigr) \\*
  &\qquad{} + \bfe_m\sum_{v\in S''} d_v\bigl(\l_v(A_v',\infty)+\log\ell_v\bigr)
       +  O\bigl(h(A)+h(\f^m)+1\bigr) \\*
    &\omit\hfill\text{from Lemma~\ref{lemma:log_chordal},} \\
  &\le \bfe_m \sum_{v\in S'}
       d_v\log\bigl|z\bigl(\f^{n-m}(P)\bigr)-z(A_v')\bigr|^{-1} \\*
  &\qquad{}+\bfe_m\sum_{v\in S} d_v\bigl(2\l_v(A_v',\infty)+\log\ell_v\bigr)
       +  O\bigl(h(A)+h(\f^m)+1\bigr).
\end{align*}
We now use Proposition~\ref{prop:heightproperty}(b,c) to observe
that
\begin{align*}
  \sum_{v\in S} d_v\l_v(A_v', \infty) 
  &\le \sum_{A'\in \f^{-m}(A)}  \sum_{v\in S} d_v\l_v(A', \infty) 
  \le \sum_{A'\in \f^{-m}(A)}  h(A')  \\
  &\le \sum_{A'\in \f^{-m}(A)}  \bigl(\hhat_\f(A') 
        +  O\bigl(h(\f) + 1\bigr) \bigr)\\
  &\le \hhat_\f(A) + O\bigl(h(\f) + 1\bigr) \bigr),
\end{align*}
Here the last line follows because there are at most~$d^m$ terms in
the sum, and $\hhat_\f(A')=d^{-m}\hhat_\f(A)$. The constants depend
only on~$m$ and~$d$, so on~$\ve$ and~$d$.  Further, from the
definition of~$\ell_v$ we have
\[
    \sum_{v\in S} d_v \log\ell_v \le  \log 2 .
\]
We also note from
Proposition~\ref{prop:polynomial-height-ineq}(d) 
that $h(\f^m)\ll h(\f)+1$, with the implied constant depending
only on $d$ and~$m$.  Hence
\begin{align}
  \label{eqn:ehhffnP}
  \ve \hhat_{\f}(\f^n(P)) 
  \le \smash[b]{
    \bfe_m \sum_{v\in S'}
  }
       d_v\log^+\bigl|z\bigl(\f^{n-m}&(P)\bigr)-z(A_v')\bigr|^{-1} \notag\\*
    &{} +   O\bigl(\hhat_\f(A)+h(\f)+1\bigr).
\end{align}
\par
We are going to apply Roth's theorem (Theorem~\ref{thm:roth})
to the set
\[
  \Upsilon = \bigl\{ z(A') : A'\in\f^{-m}(A) \bigr\}\subset\Kbar,
\]
the map $\b:S'\to\Upsilon$ given by $\b(v)=A_v'$, and the points
$x=\f^{n-m}(P)$ for $n\in\G_S(\e)$.  We note that~$\Upsilon$ is a
$G_{\Kbar/K}$-invariant set and that \text{$\#\Upsilon\le d^m$}. We apply
Theorem~\ref{thm:roth} to the set of places~$S'$, taking~$M=0$
and~$\mu=\frac52$. This gives constants~$r_1$ and~$r_2$, depending
only on~$[K:\QQ]$,~$d$, and~$\ve$, such that the set of $n\in
\G_S(\e)$ with $n>m$ can be written as a union of three sets,
\[
  \bigl\{n\in\G_S(\e):n>m\bigr\} = T_1 \cup T_2 \cup T_3,
\]
characterized as follows:
\begin{align*}
  \#T_1 &\le 4^{\#S'}r_1, \\
  T_2 &= \left\{ n>m : \sum_{v\in S'} 
    d_v \log^+\bigl|z\bigl(\f^{n-m}(P)\bigr)-z(A_v')\bigr|^{-1}
    \le \frac{5}{2} h\bigl(\f^{n-m}(P)\bigr) \right\}, \\
  T_3 &= \left\{ n>m : 
    h\bigl(\f^{n-m}(P)\bigr) \le r_2 \max_{v\in S'} \bigl\{h(A_v'),1\bigr\}
     \right\}.
\end{align*}
\par
We already have a bound for the size of~$T_1$, so we look at~$T_2$
and~$T_3$.  We start with~$T_3$ and use
Proposition~\ref{prop:heightproperty}(b,c)  to estimate
\begin{align*}
  h(A_v') 
    &\le \hhat_\f(A') +  c_3 h(\f) + c_4  \\
    &= d^{-m}\hhat_\f(A) +  c_3 h(\f) + c_4, \\[1\jot]
  h\bigl(\f^{n-m}(P)\bigr)
    &\ge \hhat_\f\bigl(\f^{n-m}(P)\bigr) - c_3 h(\f) - c_4\\
    &= d^{n-m}\hhat_\f(P) - c_3 h(\f) - c_4.
\end{align*}
Hence
\[
  T_3 \subset \bigl\{ n>m : 
    d^{n-m}\hhat_\f(P)\le c_5\hhat_\f(A)+c_6h(\f)+c_7 \bigr\},
\]
so every $n\in T_3$ satisfies
\begin{align}
  \label{eqn:T1bd}
  n &\le m 
   +  \log_d^+ \left(\frac{c_5\hhat_\f(A)+c_6h(\f)+c_7}{\hhat_\f(P)}\right) 
    \notag\\*
  &\le  c_8
   +  \log_d^+ \left(\frac{\hhat_\f(A)+h(\f)}{\hhat_\f(P)}\right).
\end{align}
\par
Finally, we consider the set~$T_2$. Again using
Proposition~\ref{prop:heightproperty}(b,c) to
relate~$h\bigl(\f^{n-m}(P)\bigr)$ to $d^{n-m}\hhat_\f(P)$, we find
that every~$n\in T_2$ satisfies
\[
  \sum_{v\in S'} 
    d_v \log^+\bigl|z\bigl(\f^{n-m}(P)\bigr)-z(A_v')\bigr|^{-1}
    \le \frac{5}{2}d^{n-m}\hhat_\f(P) + c_3h(\f) + c_4.
\]
We substitute this estimate into~\eqref{eqn:ehhffnP} to obtain
\[
  \ve \hhat_{\f}\bigl(\f^n(P)\bigr) 
  \le \bfe_m \frac{5}{2}d^{n-m}\hhat_\f(P)
    +   c_9\bigl(\hhat_\f(A)+h(\f)+1\bigr).
\]
We know from~\eqref{eqn:emk1k2dm} that $\bfe_m \le \ve d^m/5$,
and also $\hhat_{\f}\bigl(\f^n(P)\bigr) =d^n\hhat_\f(P)$, which yields
\[
  \ve d^n \hhat_\f(P)
  \le \left(\frac{\ve}{5}d^m\right)\frac{5}{2}d^{n-m}\hhat_\f(P)
    +   c_9\bigl(\hhat_\f(A)+h(\f)+1\bigr).
\]
A little bit of algebra gives the inequality
\begin{align}
  \label{eqn:T2bd}
  n &\le \log_d\left(2c_9 \frac{\hhat_\f(A)+h(\f)+1}{\ve\hhat_\f(P)}\right) 
         \notag\\
  &\le c_{10}+\log_d^+ \left(\frac{\hhat_\f(A)+h(\f)}{\hhat_\f(P)}\right).
\end{align}
Combining the estimate for~$\#T_1$ with the bounds~\eqref{eqn:T1bd}
and~\eqref{eqn:T2bd} for the largest elements in~$T_2$ and~$T_3$
completes the proof of~(a).
\par
We note that~(b) follows immediately from~(a).
\par
Finally, we prove~(c). Our first observation is that the
set~$\Upsilon=z\bigl(\f^{-m}(A)\bigr)$ used in the application of
Roth's theorem does not depend on the point~$P$. So the 
largest element in the finite set~$T_1$ is bounded independently of~$P$.
(Of course, since Roth's theorem is not effective, we do not have an
explicit bound for~$\max\Upsilon$ in terms~$K$,~$S$,~$\ve$,~$\f$ and~$A$,
but that is not relevant.)  
\par
Our second observation is to note that the quantity
\[
  \hhat_{\f,K}^{\min} \;\stackrel{\text{def}}{=}\;
  \inf\bigl\{ \hhat_\f(P) : 
    \text{$P\in\PP^1(K)$ wandering for $\f$} \bigr\}
\]
is strictly positive. To see this, let~$P_0\in\PP^1(K)$ be any
$\f$-wandering point. Then
\[
  \hhat_{\f,K}^{\min}
  = \inf\bigl\{ \hhat_\f(P) : \text{$P\in\PP^1(K)$ and
    $0 < \hhat_\f(P) \le \hhat_\f(P_0)$} \bigr\}.
\]
This last set is finite, so the infimum is over a finite set of
positive numbers, hence is strictly positive.  Therefore in the upper
bounds~\eqref{eqn:T1bd} and~\eqref{eqn:T2bd} for~$\max T_2$ and~$\max
T_3$, we may replace~$\hhat_\f(P)$ with~$\hhat_{\f,K}^{\min}$ to obtain
upper bounds that are independent of~$P$.
This proves that \text{$\max(T_1\cup T_2\cup T_3)$} may
be bounded independently of~$P$, which completes the proof of~(c).
\end{proof}

\section{A Bound for the Number of Integral Points in an Orbit}
\label{sec:bndintpts}

In this section, we use Theorem~\ref{thm:main-theorem} to give a
uniform upper bound for the number of $S$-integral points in an orbit.

\begin{corollary}
\label{cor:s-integer}
Let~$K$ be a number field, let~$S\subset M_K$ be a finite set of
places that includes all archimedean places, let~$R_S$ be the ring of
$S$-integers of~$K$, and let~$d\ge2$.  There is a
constant~$\g=\g\bigl(d,[K:\QQ]\bigr)$
such that
for all rational maps $\f \in K(z)$ of degree~$d$ satisfying
$\f^2(z) \notin K[z]$ and all $\f$-wandering points~$P\in\PP^1(K)$,
the number of $S$-integral points in the orbit of~$P$ is bounded by
\[
  \# \bigl\{ n\ge1 : z\bigl(\f^n(P)\bigr)\in R_S \bigr\}
  \le  4^{\#S}\g
      + \log_d^+ \left(\frac{h(\f)}{\hhat_\f(P)}\right).
      \]
  \end{corollary}

\begin{proof}
By definition, an element $\a\in K$ is in~$R_S$ if and only if
$|\a|_v\le1$ for all~$v\notin S$, or equivalently, if and only 
if
\[
  h(\a) = \sum_{v\in S} d_v\log\max\bigl\{|\a|_v,1\bigr\}.
\]
We note that for $v\in M_K^0$ we have
\[
  \l_v(\a,\infty) = \l_v\bigl([\a,1],[1,0]\bigr)
  = \log\max\bigl\{|\a|_v,1\bigr\}.
\]
The formula for~$\l_v$ when~$v$ is archimedean is slightly different,
but the trivial inequality $\max\{t,1\} \le \sqrt{t^2+1}$ shows that
for $v\in M_K^\infty$ we have
\[
  \log\max\bigl\{|\a|_v,1\bigr\} \le \l_v(\a,\infty).
\]
Hence
\[
  \a\in R_S \quad\Longrightarrow\quad
   h(\a) \le \sum_{v\in S} d_v\l_v(\a,\infty).
\]
\par
Let $n\ge1$ satisfy $z\bigl(\f^n(P)\bigr)\in R_S$. Then
\begin{equation}
  \label{eqn:intpt1}
  h\bigl(\f^n(P)\bigr) \le \sum_{v\in S} d_v\l_v\bigl(\f^n(P),\infty\bigr).
\end{equation}
Proposition~\ref{prop:heightproperty} tells us that
\begin{equation}
  \label{eqn:intpt2}
  h\bigl(\f^n(P)\bigr)
  \ge \hhat_\f\bigl(\f^n(P)\bigr) - c_3h(\f) - c_4
  = d^n\hhat_\f(P) - c_3h(\f) - c_4,
\end{equation}
where~$c_3$ and~$c_4$ depend only on~$d$.
Combining~\eqref{eqn:intpt1} and~\eqref{eqn:intpt2} gives
\begin{equation}
  \label{eqn:intpt3}
  \sum_{v\in S} d_v\l_v\bigl(\f^n(P),\infty\bigr)
  \ge d^n\hhat_\f(P) - c_3h(\f) - c_4.
\end{equation}
\par
We consider two cases. First, if
\[
  d^n\hhat_\f(P) \le 2 c_3h(\f) + 2 c_4,
\]
then the number of possible values of~$n$ is at most
\[
  \log_d^+\left(\frac{2 c_3h(\f) + 2 c_4}{\hhat_\f(P)}\right),
\]
which has the desired form. Second, if
\[
  d^n\hhat_\f(P) \ge 2 c_3h(\f) + 2 c_4,
\]
then~\eqref{eqn:intpt3} implies that
\begin{equation}
  \label{eqn:intpt4}
  \sum_{v\in S} d_v\l_v\bigl(\f^n(P),\infty\bigr)
  \ge \frac12 d^n\hhat_\f(P) = \frac12\hhat_\f\bigl(\f^n(P)\bigr).
\end{equation}

Now Theorem~\ref{thm:main-theorem}(b) with $\ve=\frac12$ and $A=\infty$
tells us that the number of~$n$ satisfying~\eqref{eqn:intpt4} is
at most
\begin{equation}
  \label{eqn:intpt5}
   4^{\#S}\g_3
      + \log^+_d\left(\frac{h(\f) +  \hhat_\f(\infty)}{\hhat_\f(P)}\right),
\end{equation}
where~$\g_3$  depends only on~$[K:\QQ]$ and~$d$.  (Note that
our assumption that~$\f^2(z)$ is not a polynomial is equivalent to the
assertion that~$\infty$ is not an exceptional point for~$\f$. This is
needed in order to apply Theorem~\ref{thm:main-theorem}.)  It only
remains to observe that
\[
  \hhat_\f(\infty) \le h(\infty)+c_3h(\f)+c_4
  \qquad\text{and}\qquad
  h(\infty)=h\bigl([0,1]\bigr)=0
\]
to see that the bound~\eqref{eqn:intpt5} has the desired form.
\end{proof}


\end{document}